\documentclass{amsart}

\usepackage{amsfonts,amsmath,amssymb}
\usepackage{bm}
\usepackage[OMLmathrm,OMLmathbf]{isomath}
\usepackage[english]{babel}
\usepackage[utf8]{inputenc}
\input xy
\xyoption{all}
\newtheorem{thm}{Theorem}
\newtheorem{cor}[thm]{Corollary}

\usepackage{graphicx}
\usepackage{bm}
\usepackage{amssymb}

\date{}

	\usepackage{xparse}
\makeatletter
\RenewDocumentCommand{\title}{om}{%
   \IfNoValueTF{#1}
     {\gdef\shorttitle{Fixed point theorems for topological contractions}}
     {\gdef\shorttitle{#1}}%
   \gdef\@title{#2}%
}
\makeatother

\title{Fixed point theorems for topological contractions and the Hutchinson operator}
\address{Micha{\l} Morayne: Faculty of Fundamental Problems of Technology, Wroc{\l}aw University of Science and Technology, 50-370 Wrocław, Poland.}
\address{Robert Ra{\l}owski: Faculty of Pure and Applied Mathematics, Wrocław University of Science and Technology, 50-370 Wroc{\l}aw, Poland.}
\author{Micha{\l} Morayne}
\author{Robert Ra{\l}owski}
\thanks{The work of Robert Rałowski has been partially financed by  grants 8211204601, 8211104160, MPK: 9120730000 from the Faculty of Pure and Applied Mathematics, Wrocław University of Science and Technology}
\email{michal.morayne@pwr.edu.pl}
\email{robert.ralowski@pwr.edu.pl}

\keywords{fixed point; $T_1$ space; compact space; attractor, contraction, Hutchinson operator}
\subjclass[2010]{54D10, 54D30, 54E52, 54H25, 54H20}

\begin{document}

\begin{abstract}
For a topological space $X$ a topological contraction on $X$ is a closed mapping $f:X\to X$ such that for every open cover of $X$ there is a positive integer $n$ such that the image of the space $X$ via the $n$th iteration of $f$ is a subset of some element of the cover.
Every topological contraction in a compact $T_1$ space has a unique fixed point. As in the case of metric spaces and the classical Banach fixed point theorem, this analogue of Banach's theorem is true not only in compact but also in complete (here in the sense of \v{C}ech) $T_1$ spaces.  We introduce a notion of weak topological contraction and in Hausdorff spaces we prove the existence of a unique fixed point for such continuous and closed mappings without assuming completeness or compactness of the space considered. These theorems are applied to prove existence of fixed points for mappings on compact subsets of linear spaces with weak topologies and for compact monoids. We also prove some fixed point results for $T_1$ locally Hausdorff spaces and, introduced here, peripherally Hausdorff spaces. An iterated function system on a topological space, IFS, is a finite family of closed mappings from the space into itself. It is contractive if for every open cover of $X$ for some positive integer $n$ the image of $X$ via a composition of $n$ mappings from the IFS is contained in an element of the cover. We show that in $T_1$ compact topological spaces the Hutchinson operator of a contractive IFS 
may not be closed as the mapping in the hyperspace of closed subsets of the space. Nevertheless, the Hutchinson operator of a contractive IFS has always a unique fixed point.
\end{abstract}

\maketitle

\section{Introduction} 
The classical Banach contraction theorem states that Lipschitz mappings with Lipschitz constants smaller than $1$ on a complete metric space, thus in particular on a compact metric space, to the same space have a unique fixed point. In \cite{MMRR} this theorem has been extended to  $T_1$ spaces, where the notion of a Lipschitz contraction, specific for metric spaces, were generalized by a notion of topological contraction. As \v{C}ech completeness, or actually a condition characterizing its original definition (e.g. \cite{E}) extends the notion of completeness to topological spaces one can ask a natural question if the above mentioned analogue of Banach's fixed point theorem for $T_1$ compact spaces holds in $T_1$ \v{C}ech complete spaces. It turns out that it does and it is a content of Theorem \ref{Banach-T1-complete}. One also could expect that in the case of $T_2$ spaces some assumptions of an analogous theorem can be weekend. We prove two such fixed point theorems for weaker notions of contractions, so called weak topological contractions with the assumptions that they are continuous (Theorem \ref{T2-continuous-fixed}) or closed (Theorem \ref{T2-closed-fixed}); in this second case we require that the space be first countable. These theorems are applied to obtain a fixed point theorems for mappings on compact subsets of linear  spaces with weak topologies.

In this article we consider special $T_1$ spaces which are more general than the Hausdorff spaces, the locally Hausdorff $T_1$ spaces. While the fixed point theorem for continuous weak topological contractions extends to this class of spaces the theorem for closed weak topological contractions does not. We also introduce a class of $T_1$ peripherally Hausdorff spaces. Also for these class of spaces and for a stronger notion of contraction, a weak+ topological contraction, we obtain a fixed point theorem           

An iterated function system on a topological space $X$, an IFS, is a finite family of closed mappings from the space into itself (we do not assume their continuity here). It is contractive if for every open cover of $X$ for some positive integer $n$ the image of $X$ via a composition of $n$ mappings from the IFS is contained in an element of the cover. Given an IFS the Hutchinson operator generated by this IFS maps a set to the union of its images via all elements of the IFS. It was shown in \cite{MMRR} that in a $T_1$ compact space $X$ the Hutchinson operator generated by a contractive IFS, if it is closed, is a topological contraction on the hyperspace of all nonempty closed subsets of $X$. Applying the analogue of the Banach fixed point theorem to the hyperspace and such an IFS and its Hutchinson operator we conclude that it has a unique fixed point which is also called an attractor.
We show here that in a $T_1$ compact topological space the Hutchinson operator of a contractive IFS 
may not be closed in the hyperspace, but, nevertheless, the Hutchinson operator of a contractive IFS has always a unique fixed point.

The paper is organized as follows. In the second section we recall the basic notations and terminology. In the third section we state and prove an analogue of the Banach fixed point theorem for $T_1$ \v{C}ech complete spaces. In the fourth section we prove two analogues of the Banach fixed point theorem for Hausdorff spaces. In the fifth section as an application of the results in the previous section we state and prove fixed point theorems for compact subsets of linear spaces with weak topologies and for compact monoids. In the next two sections we consider $T_1$ locally Hausdorff spaces and $T_1$ peripherally Hausdorff spaces and prove fixed point theorems for these classes of spaces. The last eighth section is devoted to the Hutchinson operator in $T_1$ compact spaces.

\section{Notations and terminology}\label{terminology}

All topological spaces considered in this article are assumed to be non-empty. 
A Tychonoff topological space is \v{C}ech complete if it is a $G_\delta$ subset of any of its compactifications. An equivalent  definition says that a Tychonoff topological space is \v{C}ech complete 
if there exists a sequence $(\mathcal{U}_i)$ of its open covers such that if a sequence $(F_m)_m$ of closed sets is centered and for each $i$ there exists $m_i$ such that $F_{m_i}$ is a subset of some member of $\mathcal{U}_i$, then the intersection $\bigcap_m F_m$ is nonempty (\cite{E}). The condition that the space is Tychonoff in the second definition is added to retain this condition from the first definition. However the second definition makes sense if we omit this assumption. Further we call a topological space {\it \v{C}ech complete} if it satisfies the second definition without assuming that it is Tychonoff. 

A metric space $(X,d)$ is an {\it Atsuji space} if for each open cover $\mathcal{U}$ of $X$  there exists $\varepsilon>0$ such that for every $x\in X$ the open ball $B_d(x,\varepsilon)$ (with radius $\varepsilon$ and center at $x$)  is contained in some element $U$ of the cover (comp. \cite{A}). It is a classical fact, called the Lebesgue number lemma, that each compact metric space is an Atsuji space (\cite{M}).  

The class of all ordinal numbers will be denoted by On.

As usual $\mathbb{N}$ denotes the set of positive integers, $\mathbb{N}_0:=\mathbb{N}\cup\{0\}$, and $\mathbb{Z}$ denotes the set of all integers.

A topological space $X$ is {\it locally Hausdorff} if every point of the space has an open  neighbourhood $U$ such that the topology of $X$ restricted to $U$ is Hausdorff. 

A topological space $X$ is $1$-Hausdorff if it is $T_1$ and for every point $x\in X$ the topology of $X$ resticted to the set
consisting of $x$ and of those points $y$ that cannot be separated from $x$ by two open disjoint neighbourhoods $U\ni x$, $V\ni y$, namely the set 
$$[x]:=\bigcap\{\overline{U}:U\;\textrm{is\;an\;open\;neigbourhood\;of}\;x\},$$
is Hausdorff. We extend this definition to so called $\alpha$-Hausdorff spaces, $\alpha\in\textrm{On}$. We define these classes by recursion with respect to ordinals. 
A topological space is a {\it $0$-Hausdorff space} space if it is a Hausdorff space. Let us assume that $0<\alpha\in\textrm{On}$ and $\beta$-Hausdorff spaces have been defined for all $\beta<\alpha$.
A topological $T_1$ space $X$ is an {\it $\alpha$-Hausdorff-space}
if for every $x\in X$ the space $[x]$ with the relative topology inherited from the whole space $X$ is a {\it $\beta_x$-Hausdorff space}
where $\beta_x<\alpha$ ($\beta_x$ depends on $x$ here). 
If $X$ is an $\alpha$-Hausdorff space for some ordinal $\alpha$, then $X$ will be called a {\it peripherally Hausdorff} space. 
It is easy to notice that the classes of $\alpha$-Hausdorff spaces form a non-decreasing trans-finite sequence 
with respect to ordinals $\alpha$.

Let  $X$ be a peripherally Hausdorff space. We define a {\it Hausdorff-rank} of $X$: 
$$\textrm{rank}_{T_2}(X)=\min\{\alpha:X\;\textrm{is\;an}\;\alpha\textrm{-}\textrm{Hausdorff}\;\textrm{space}\}.$$

In Section \ref{alfa-H} we present some basic properties of $\alpha$-Hausdorff spaces.

Let $X$ be a fixed set. Let $f:X\to X$. Then $f^n$ denotes the composition of $n$ copies of $f$: $f^n:=f\circ f\circ \ldots \circ f$; $f^0$ denotes the identity mapping on $X$.

Let $X$ be a topological space. A mapping $f:X\to X$ is called {\it a topological contraction} if $f$ is a closed mapping and for every open cover $\mathcal{U}=\{U_\alpha:\alpha\in\Lambda\}$ there exist $n\in \mathbb{N}$ and $\alpha\in\Lambda$ such that $f^n[X]\subseteq U_\alpha$
 (\cite{MMRR}). We do not require here that $f$ be continuous. If $X$ is a compact $T_1$ space then (Theorem 4 in \cite{MMRR}), $f:X\to X$ is a topological contraction if and only if $f$ is closed and

\vspace{0.2 cm}

\noindent $(\star)$ for each two different points $x,y\in X$ there exists $n\in\mathbb{N}$ such that
$f^n[X]\subseteq X\setminus\{x\}$ or $f^n[X]\subseteq X\setminus\{y\}$.

\vspace{0.2 cm}

If $X$ is a topological space then a mapping $f:X\to X$ will be called {\it a weak topological contraction}  if for every open cover $\mathcal{U}=\{U_\alpha:\alpha\in\Lambda\}$ of $X$ and every pair of points $x,y\in X$ there exist $n\in \mathbb{N}_0$ and $\alpha\in\Lambda$ such that $f^n[\{x,y\}]\subseteq U_\alpha$ (\cite{MMRR}). In this definition we do not require continuity or closedness of $f$. We do so because we will consider both versions, continuous and closed, of weak topological contractions.

If $X$ is a topological space then a mapping $f:X\to X$ will be called {\it a weak+ topological contraction}  if for every open cover $\mathcal{U}=\{U_\alpha:\alpha\in\Lambda\}$ of $X$ there exist $n_
0\in \mathbb{N}_0$ and $\alpha\in\Lambda$ such that $f^n[\{x,y\}]\subseteq U_\alpha$ for all $n\geq n_0$. Also in this definition we do not require continuity or closedness of $f$.

For a topological space $X$, an {\it iterated function system} (for short: IFS) is any finite family  $\mathcal{F}=\{f_1,\ldots,f_m\}$ of closed mappings from $X$ to $X$. We say that an IFS $\mathcal{F}=\{f_1,\ldots,f_m\}$ is {\it contractive} if 
for any open cover $\mathcal{U}=\{U_\alpha:\alpha\in\Lambda\}$ of $X$ there exists $n\in\mathbb{N}$ such that for any sequence $(i_1,\ldots,i_n)$, $1\leq i_j\leq m$, the set $f_{i_1}\circ\ldots\circ f_{i_n}[X]$ is contained in some element $U_\alpha$ of the cover   $\mathcal{U}$.  Thus $f:X\to X$ is a topological contraction if the one element IFS consisting of the mapping $f$ is contractive. 

Let $X$ be a $T_1$ topological space. Let $2^X$ be the family of all closed nonempty 
subsets of $X$.  The Vietoris topology on $2^X$ is generated by the basis consisting of all sets of the form 
$$S(V_0;V_1,\ldots,V_k):=\{K\in2^X:K\subseteq V_0, K\cap U_i\neq\emptyset,1\leq i\leq k\},$$
where $k\in\mathbb{N}$ and $V_0,V_1,\ldots,V_k$ are open subsets of $X$. The space $2^X$ with the Vietoris topology is called the {\it hyperspace} of $X$. It is known that if $X$ is a $T_1$ compact space then so is its hyperspace $2^X$ (\cite{Mi}). 

Let $\mathcal{F}=\{f_1,\ldots,f_m\}$ be an IFS on a $T_1$ space $X$. The {\it Hutchinson operator $F: 2^X \to 2^X$ induced by $\mathcal{F}$} 
is defined by the formula 
$$F(K):=\bigcup_{i=1}^mf_{i}[K]$$
(see \cite{H}). If $K\in2^X$ is a fixed point of the Hutchinson operator induced by an IFS $\mathcal{F}$ we call $K$ an {\it attractor} of  $\mathcal{F}$.

The original theorems proved in this paper will be numbered while the theorems cited from other sources will be lettered.

\section{An analogue of the Banach fixed point theorem in $T_1$ \v{C}ech complete spaces}

The following theorem was proven in  \cite{MMRR} (Theorem 3 therein).

\vspace{0.4 cm}

\noindent{\bf Theorem A.} {\it
If $X$ is a $T_1$ compact space, $f:X\to X$ is a topological contraction, then 
there exists a unique fixed point for the mapping $f$.} 

\vspace{0.4 cm}

Theorem A generalizes the Banach fixed point theorem for Lipschitz contractions in metric compact spaces. However, the classical version of this theorem is stated for complete metric spaces. Here we will give a $T_1$ analogue of this theorem which generalizes the classical theorem for bounded complete metric spaces.    


\begin{thm}\label{Banach-T1-complete}
If $X$ is a $T_1$ \v{C}ech complete space and $f:X\to X$ is a topological contraction, then $f$ has a unique fixed point.
\end{thm}
  
\begin{proof}  Let $(\mathcal{U}_i)_i$ be a sequence of open covers of $X$ from the definition of \v{C}ech completeness. The sequence $(f^n[X])_n$ is a non-increasing sequence of non-empty closed sets ($f$ is closed mapping); thus it is centered. Because $f$ is a topological contraction for every $i\in\mathbb{N}$ there exists $n_i$ such that $f^{n_i}[X]$ is contained in some element of $\mathcal{U}_i$. Hence the intersection
$\bigcap_nf^n[X]$ is nonempty. We claim that there is only one $x\in\bigcap_nf^n[X]$. Assume that $y\in\bigcap_nf^n[X]$ and $y\neq x$. The two element family $\{X\setminus\{x\},X\setminus\{y\}\}$ is an open cover of $X$ and for no $n$ $f^n[X]$ is contained in any of these two members of the cover; this contradicts the assumption that $f$ is a topological contraction. 

Now let $x\in\bigcap_nf^n[X]$. Because then $f(x)\in\bigcap_nf^n[X]$ it follows that $f(x)=x$ and $x$ is a fixed point of $f$. It is a unique fixed point because if $f(y)=y$ for any $y\in X$ we have $y\in\bigcap_nf^n[X]$ and then, as we noticed above, $y=x$. This concludes the proof.  
\end{proof}

\section{Analogues of the Banach fixed point theorem in Hausdorff spaces}

We shall prove two fixed-point theorems for continuous weak topological contractions and closed weak topological contractions on Hausdorff topological spaces.

\begin{thm}\label{T2-continuous-fixed}
If $X$ is a Hausdorff topological space and $f$ is a continuous weak topological contraction on $X$, then $f$ has a unique fixed point.
\end{thm}

\begin{proof} Let $x_0\in X$ be fixed. We claim that there exists a point $y\in X$ such that for each neighbourhood $U$ of $y$ there exists 
$n\in \mathbb{N}_0$ such that $\{f^n(x_0),f^{n+1}(x_0)\}\subseteq U$. If it were not the case then  every point $x\in X$ would have a neghbourhood $V_x$ not containing any set $\{f^n(x_0),f^{n+1}(x_0)\}=\{f^n(x_0),f^{n}(f(x_0))\}$ and, for the open cover $\{V_x:x\in X\}$, this would contradict $f$ being a weak topological contraction 

We claim that $y$ is a fixed point of $f$. Assume, striving for a contradiction, that $y\neq f(y)$. Let $U,V$ be disjoint 
neighbourhoods of $y$ and $f(y)$, respectively. Let $W=f^{-1}[V]\cap U$. The set $W$ is a neighbourhood of $y$, hence it contains  
the set  $\{f^n(x_0),f^{n+1}(x_0)\}$ for some $n$. But then 
 $\{f^{n+1}(x_0),f^{n+2}(x_0)\}\subseteq V$ and $f^{n+1}(x_0)\in U\cap V$ which is a contradiction with our assumption that
 $U,V$ are disjoint.

The point $y$ is the unique fix point of $f$. Indeed, assume that $z\neq y$ and $z$ is also a fixed point of $f$. Then 
$\{X\setminus\{y\},X\setminus\{z\}\}$ is an open cover of $X$ but for no $n$ the set $\{f^n(y),f^n(z)\}=\{y,z\}$ is contained in a member of this cover.  

\end{proof}

One can notice by inspection of the proof of Theorem \ref{T2-continuous-fixed} that much weaker assumption is sufficient to obtain the existence of a fixed point, but maybe without its uniqueness. Namely the following theorem holds with the proof going along the same lines as that of Theorem \ref{T2-continuous-fixed}  (except the uniqueness part).

\begin{thm}\label{T2-continuous-fixed-weak}
If $X$ is a Hausdorff topological space, $f:X\to X$ is a continuous mapping,  there exists $x_0\in X$ such that for each open cover $\mathcal{U}$ of $X$ there is $n\in\mathbb{N}_0$ such that $\{f^n(x_0),f^{n+1}(x_0)\}\subseteq U$ for some $U\in\mathcal{U}$, then $f$ has a fixed point.   
\end{thm}

From the above theorem one can derive the following corollary, which is a theorem due to G. Beer (Theorem 4 in \cite{B}).

\begin{cor}\label{Beer}\textnormal{(Beer)}
Let $(X,d)$ be an Atsuji space and $f:X\to X$ be a continuous mapping. If there exists an $x_0\in X$ such that $\liminf_{n\to\infty}d(f^n(x_0),f^{n+1}(x_0))=0$, then $f$ has a fixed point.   
\end{cor}

\begin{proof}
Let $\mathcal{U}$ be an open cover of $X$. Because $X$ is an Atsuji space there exists $\varepsilon>0$ such that each open ball $B_d(y,\varepsilon)$ is contained in some $U\in\mathcal{U}$. From $\liminf_{n\to\infty}d(f^n(x_0),f^{n+1}(x_0))=0$ we infer that there exists $n\in \mathbb{N}_0$ such that  $d(f^n(x_0),f^{n+1}(x_0))<\varepsilon$. Hence $\{f^n(x_0),f^{n+1}(x_0)\}\subseteq B_d(f^n(x_0),\varepsilon)\subseteq U$ for some $U\in \mathcal{U}$ and the hypothesis of Theorem \ref{T2-continuous-fixed-weak} is satisfied and, by this theorem, $f$ has a fixed point.
\end{proof}  

The following theorem is a counterpart of Theorem \ref{T2-continuous-fixed} for closed mappings (instead of continuous). 

\begin{thm}\label{T2-closed-fixed}
If $X$ is a Hausdorff first-countable topological space and $f$ is a closed weak topological contraction on $X$, then $f$ has a unique fixed point.
\end{thm}

\begin{proof} Let $x_0\in X$ be fixed. As in the proof of Theorem \ref{T2-continuous-fixed} we show that
there exists a point $y\in X$ such that for each neighbourhood $U$ of $y$ there exists 
$n\in \mathbb{N}_0$ such that $\{f^n(x_0),f^{n+1}(x_0)\}\subseteq U$. We claim that $f(y)=y$. Striving for a contradiction, let us assume that $f(y)\neq y$. 
Because $X$ is first-countable we can assume that there exists a sequence $(f^{n_i}(x_0))_i$ tending to $y$ such that the sequence $(f^{n_i+1}(x_0))_i$ also tends to $y$. We can also assume that $f^{n_i}(x_0)\neq y$ for all $i\in\mathbb{N}$ (if $f^{n_i}(x_0)$ is equal to $y$  for infinitely many $i$s , then $f^{n_i+1}(x_0)=f(y)$ and, because $f^{n_i+1}(x_0)\to y$, we get $f(y)=y$). If $f^{n_i+1}(x_0)=y$ for infinitely many $i$s then $f^{n_i+1}(x_0)=y$ and $f^{n_j+1}(y)=y$ for some $j>i$ and the points 
$f^{n_i+1}(x_0), f^{n_i+2}(x_0),\ldots$
would form a cycle of length at most $n_j-n_i$. Hence, by the convergence $f^{n_k}(x_0)\to y$, we get  $f^{n_k}(x_0)= y$ for all $k$'s greater than some $k_0$, and then, as we argued above, $f(y)=y$. Thus we can also assume that all the points $f^{n_i+1}(x_0)$ are different than $y$. The set $D:=\{f^{n_i}(x_0):i\in\mathbb{N}\}\cup\{y\}$ is closed and its image $f[D]=\{f^{n_i+1}(x_0):i\in\mathbb{N}\}\cup\{f(y)\}$ is not closed because it does not contain $y$ which is the limit of the sequence $(f^{n_i+1}(x_0))_i$. This contradicts the assumption that $f$ is a closed mapping.

Assume now that $z\neq y$ is also a fixed point of $f$. Then for the open cover $\{X\setminus\{y\},X\setminus\{z\}\}$ for no $n$ the set $f^n[\{y,z\}]=\{y,z\}$ is contained in an element of this cover. This contradicts the assumption that $f$ is a weak contraction.

\end{proof}

As for the case of continuous mappings one can notice by inspection of the proof of Theorem \ref{T2-closed-fixed} that again a much weaker assumption is sufficient to obtain the existence of a fixed point, maybe without its uniqueness. Namely, the following theorem holds with the proof going along the same lines as that of Theorem \ref{T2-closed-fixed} (except the uniqueness part).

\begin{thm}\label{T2-closed-fixed-weak}
If $X$ is a first countable Hausdorff topological space, $f:X\to X$ is a closed mapping,  there exists $x_0\in X$ such that for each open cover $\mathcal{U}$ of $X$ there is $n\in\mathbb{N}_0$ such that $\{f^n(x_0),f^{n+1}(x_0)\}\subseteq U$ for some $U\in\mathcal{U}$, then $f$ has a fixed point.   
\end{thm}

From this theorem we can derive a version of Beer's theorem (Corollary \ref{Beer}) for closed mapping on Atsuji spaces; we note that, as they are metric spaces they are first countable.

\begin{cor}\label{Beer-closed}
Let $(X,d)$ be an Atsuji space and $f:X\to X$ be a closed mapping. If there exists an $x_0\in X$ such that $\liminf_{n\to\infty}d(f^n(x_0),f^{n+1}(x_0))=0$, then $f$ has a fixed point.   
\end{cor}

One can ask if  Theorems \ref{T2-continuous-fixed} and \ref{T2-closed-fixed} hold as well for $T_1$ compact spaces. The following example shows that this is not the case. The same example shows that in the hypothesis of Theorem \ref{Banach-T1-complete} "topological contraction" cannot be replaced by "weak topological contraction".

\vspace{0.4 cm}

\noindent{\bf Example 1.} Let $\mathbb{N}$ be endowed with the topology consisting of all co-finite subsets of $\mathbb{N}$. The space is compact and therefore it is also \v{C}ech complete. 
Let $f:\mathbb{N}\to\mathbb{N}$ be a permutation of $\mathbb{N}$ without finite cycles. As a permutation of  $\mathbb{N}$ $f$ is both continuous and closed. It is also a weak topological contraction. Indeed, if $k,m\in\mathbb{N}$ then, because $f$ has no finite cycles,  $\lim_{n\to\infty}f^n(k)=\lim_{n\to\infty}f^n(m)=\infty.$. Thus for each open set $U$, which is a co-finite subset of $\mathbb{N}$, for $n$ large enough we have
$f^n[\{m,n\}]\in U$. The mapping $f$ does not have cycles of finite length, so in particular of length one, in other words it does not have a fixed point.

\vspace{0.4 cm}

The next example shows that the assumption from Theorem \ref{T2-closed-fixed} that the space considered is first countable cannot be dropped. 

\vspace{0.4 cm}

\noindent{\bf Example 2.} Let $A\subseteq\mathbb{N}$. The upper density $\bar{\delta}(A)$ of $A$ is defined as  
$$\bar{\delta}(A):=\limsup_n\frac{|\{1,\ldots,n\}\cap A|}{n}=\alpha,$$
and the the density $\delta(A)$  of $A$ as
$${\delta}(A):=\lim_n\frac{|\{1,\ldots,n\}\cap A|}{n}=\alpha$$
if this limit exists.

Let $r\in\mathbb{Z}$. Let $A\subseteq\mathbb{N}$. Let
$$A+r:=\{a+r:a\in A\}\cap\mathbb{N}.$$
Let $\mathcal{G}$ be a family of subsets of $\mathbb{N}_0$. Let 
$$\mathcal{G}+r:=\{G+r:G\in \mathcal{G}\},$$
and
$$S(\mathcal{G}):=\{\mathcal{G}-n:n\in\mathbb{N}_0\}.$$

Let $\mathcal{A}$ be the family of all subsets of $\mathbb{N}$ that have density equal to one.
We will construct a family of ultrafilters $\{\mathcal{F}_\alpha:\alpha<\mathfrak{c}\}$ such that the classes
$S(\mathcal{F}_\alpha)$ are pairwise disjoint, each ultrafilter contains a subset of $\mathbb{N}$ of density zero,
and each infinite subset of $\mathbb{N}$ is contained in one of the ultrafilters  $\mathcal{F}_\alpha.$

We construct the family $\{\mathcal{F}_\alpha\:\alpha<\mathfrak{c}\}$ by a transfinite recursion of length $\mathfrak{c}$ .
Let $\{E_\alpha:\alpha<\mathfrak{c}\}$ be the family of all infinite subsets of $\mathbb{N}$.
Let us assume that we have already defined ultrafilters $\mathcal{F}_\alpha$ for $\alpha<\gamma$. Let 
$$\xi=\min\{\beta:E_\beta\notin\bigcup_{\alpha<\gamma}\mathcal{F}_\alpha\}.$$
The set from which we take minimum is not empty because $\mathbb{N}$ is a union of $\mathfrak{c}$ many almost disjoint infinite sets and 
no two of them can belong to the same ultrafilter but so far we have defined less than $\mathfrak{c}$ ultrafilters so one of these almost disjoint infinite sets does not belong to any of them. The set $E_\xi$ contains an infinite set $F$ of density zero. 
As there are $2^{\aleph_0}$ many ultrafilters containing a given infinite set
we can find an ultrafilter  $\mathcal{F}_\gamma$ that does not belong to $\bigcup_{\alpha<\gamma}S(\mathcal{F}_\alpha)$ containing $F$.
Then it also contains $E_\xi$.

Now let 
$$X=\mathbb{N}\cup\{\infty\}\cup\bigcup_{\alpha<\mathfrak{c}}S(\mathcal{F}_\alpha).$$
We define a topology on $X$ in the following way. The neighbourhoods of $\infty$ are sets $A\cup\{\infty\}$, $A\in\mathcal{A}$. On the set $X\setminus\{\infty\}$ the topology is inherited from the Wallman compactification $\beta\mathbb{N}$ of $\mathbb{N}$.
This space is $T_2$. The only nontrivial item to check is separation of ultrafilters $\mathcal{F}_\alpha-n$ from $\infty$. It follows from the fact that each such ultrafilter has as an element a set $P$ of density zero and then the neighbourhood $\mathbb{N}\setminus P$ of $\infty$ does not intersect the neighbourhood of the ultrafilter generated by $P$.

Let $f:X\to X$ be defined in the following way: $f(n)=n+1$ for $n\in\mathbb{N}$, $f(\infty)=1$, $f(\mathcal{F}_\alpha)=\infty$ and
$f(\mathcal{F}_\alpha-n)=\mathcal{F}_\alpha-(n-1)$, $\alpha<\mathfrak{c}$ , (of course, $\mathcal{F}_\alpha-0=\mathcal{F}_\alpha$). We claim that $f$ is a closed weak contraction. Let us check first that $f$ is closed. Let $D$ be a closed subset of $X$.  

Let us assume that $\mathcal{F}_\alpha-m$ is an accumulation point of $f[D]$. Thus it is also an accumulation point of $f[D]\setminus\{\infty\}$. The mapping
$$f|_{X\setminus(\{\infty\}\cup\{\mathcal{F}_\alpha: \alpha<\mathfrak{c}\})}:X\setminus(\{\infty\}\cup\{\mathcal{F}_\alpha:\alpha<\mathfrak{c}\})\to X\setminus\{\infty,1\}$$
is a homeomorphism. Thus 
$$f[D\cap(X\setminus(\{\infty\}\cup\{\mathcal{F}_\alpha:\alpha<\mathfrak{c}\}))]=f[D]\setminus\{\infty,1\}$$
is closed in $X\setminus\{\infty,1\}$ and
it contains each accumulation point of $f[D]\setminus\{\infty,1\}$, thus it also contains $\mathcal{F}_\alpha-m$.  

Next let us assume that $\infty$ is an accumulation point of $f[D]$. Thus in every set $A$ from $\mathcal{A}$ there lies a point from $f[D]$.
Hence in every set $A+1$, $A\in\mathcal{A}$, there lies a point from $f[D]$ because 
$\mathcal{A}+1\subseteq\mathcal{A}$. Hence in every set $A$ from $\mathcal{A}$ there lies a point from $(f[D]\cap\{2,3,\ldots\})-1\subseteq D$. It follows that $D\cap\mathbb{N}$ has positive upper density and, by our construction of the family $\{\mathcal{F}_\alpha:\alpha<\mathfrak{c}\}$, is a member of the ultrafilter $\mathcal{F}_\alpha$, for some $\alpha<\mathfrak{c}$.
Hence in every neighbourhood $U$ of $\mathcal{F}_\alpha$ there are points of $D$. Then, $\mathcal{F}_\alpha\in D$ because $D$ is closed, and we obtain $\infty=f(\mathcal{F}_\alpha)\in f[D].$

To conclude the example we show that $f$ is a weak contraction. Let $U$ be any open set containing $\infty$. It is enough to show that for any two points $x_1,x_2$ from $X$ there exists $n\in\mathbb{N}$ such that $\{f^{n}(x_1),f^{n}(x_2)\}\subseteq U$. From the definition of $f$ we easily see that $\{f^{k}(x_1),f^{k}(x_2)\}\subseteq \mathbb{N}$ for some positive integer $k$. Then, because $U$ contains some set $A\subseteq\mathbb{N}$ of density equal to one, there must be some positive integer $m$ such that
$$\{f^{k+m}(x_1),f^{k+m}(x_2)\}=\{f^{k}(x_1)+m,f^{k}(x_2)+m\}\subseteq A\subseteq U$$
(because, otherwise, we would have $\bar{\delta}(A)\leq1-\frac{1}{k}$).

\section{An application: a fixed point theorem for weak topologies }\label{linear spaces}

As an application of Theorem \ref{T2-continuous-fixed} we will prove a fixed point theorem for compact sets in linear topological spaces with topologies generated by familes of seminorms. In particular it applies to compact sets in the weak* topologies. In some cases these compact sets are not metrizable and one cannot apply the classical Banach theorem. In the completely metrizable cases the theorems of this section provide an alternative (to applying Banach's theorem method) to prove the existence of a fixed point.

Let us define one more notion. Let $X$ be a linear topological space, $\mathbf{E}=\{\mathbf{s}_1,\ldots,\mathbf{s}_n\}$ be a finite set of seminorms on $X$, $x\in X$ and $\varepsilon>0$. Let
$$V(x;\mathbf{E},\varepsilon):=\{z\in X:\mathbf{s}_1(x-z)<\varepsilon,\ldots,\mathbf{s}_n(x-z)<\varepsilon\}.$$

We start with the following theorem which is an analogue of Lebesgue number lemma (\cite{M}).

\begin{thm}\label{LN}
If $X$ is a linear topological space, $\mathbf{S}$ is a family of seminorms separating points in $X$, $Y\subseteq X$ is compact in the weak topology $\tau$ determined by $\mathbf{S}$, and $\mathcal{U}$ is
a $\tau$-open cover of $Y$, then there exist a finite set $\mathbf{E}\subseteq \mathbf{S}$ and $\varepsilon\in\mathbb{R}$ such that for each $x\in Y$ there exists $U\in\mathcal{U}$ such that $V(x;\mathbf{E},\varepsilon)\subseteq U$.   
\end{thm} 

\begin{proof}  Striving for a contradiction, let us assume that the conclusion does not hold. As $Y$ is compact there exists a finite subcover $\mathcal{U}_1\subseteq\mathcal{U}$.
By taking a refinement of $\mathcal{U}$ consisting of base sets of the form $V(y;\mathbf{F},\delta)$, where $y\in Y$, $\mathbf{F}\subseteq \textbf{S}$, $|\mathbf{F}|<\aleph_0$, $\delta\in(0,\infty)$, we can assume that  
$$\mathcal{U}_1=\{V(y_1;\mathbf{E}_1,\delta_1),\ldots,V(y_n;\mathbf{E}_n,\delta_n)\}.$$
Because the conclusion does not hold, for every $k\in\mathbb{N}$ there exists $x_k\in Y$ such that $V(x_k;\bigcup_{i=1}^n\mathbf{E}_i,1/k)\not\subseteq V(y_m;\mathbf{E}_m,\delta_m)$ for each $1\leq m\leq n$. Let $x\in Y$ be an accumulation point of the sequence
$(x_n)_n$ with respect to the topology $\tau$; such a point exists because $Y$ is compact.Thus for every $\gamma>0$ infinitely many elements of the sequence $(x_k)_k$ lie in $V(x;\bigcup_{i=1}^n\mathbf{E}_i,\gamma)$. Let $z_{k,m}\in V(x_k;\bigcup_{i=1}^n\mathbf{E}_i,1/k)\setminus V(y_m;\mathbf{E}_m,\delta_m)$. Hence there exists $\mathbf{s}\in \mathbf{E}_m$ such that $\mathbf{s}(z_{k,m}-y_m)\geq\delta_m$. Let $0<\delta'_m<\delta_m$. Let $\frac{1}{k}+\gamma<\mathbf{s}(z_{k,m}-y_m)-\delta'_m$.
Then we have for every $1\leq m\leq n$
$$\mathbf{s}(x-y_m)\geq\mathbf{s}(z_{k,m}-y_m)-\mathbf{s}(z_{k,m}-x_k)-\mathbf{s}(x_k-x)>$$
$$\mathbf{s}(z_{k,m}-y_m)-\left(\frac{1}{k}+\gamma\right)>\delta'_m.$$
Hence, as $\delta'_m<\delta_m$ is any positive number arbitrarily close to $\delta_m$, we obtain
$$\mathbf{s}(x-y_m)\geq\delta_m.$$
Thus $x\notin V(y_m;\mathbf{E}_m,\delta_m)$ for every $1\leq m\leq n$, but this is a contradiction with $x\in Y$. 

\end{proof} 

\begin{thm}\label{conv-fixed} If $X$ is a linear topological space, $\mathbf{S}$ is a family of seminorms on $X$ separating points in $X$, $Y\subseteq X$ is compact in the topology $\tau$ generated by $\mathbf{S}$, and $f:Y\to Y$ is a continuous mapping in the topology $\tau$ such that
\begin{equation}\label{conv}
\lim_n\mathbf{s}(f^n(x)-f^n(y))=0
\end{equation}
for each $\mathbf{s}\in\mathbf{S}$  and each $x,y\in Y$, then $f$ has a unique fixed point in $Y$. 
\end{thm}

\begin{proof} Let $y_1,y_2\in Y$. Let $\mathcal{U}$ be a $\tau$-open cover of $Y$. Let $\varepsilon>0$ and $\mathbf{E}\subset \mathbf{S}$, $|\mathbf{E}|<\aleph_0$, be chosen for $\mathcal{U}$ accordingly to Theorem \ref{LN}. Thus for every $x\in Y$ the set $V(x;\mathbf{E},\varepsilon)$ is a subset of some set $U$ from the cover $\mathcal{U}$. Let 
$n$ be chosen to guarantee $\mathbf{s}(f^n(y_1)-f^n(y_2))<\varepsilon$ for $\mathbf{s}\in\mathbf{E}$; we can choose such an $n$ because the set $\mathbf{E}$ is finite. Let $V(f^n(y_1);\mathbf{E},\varepsilon)\subseteq U$ for some $U\in\mathcal{U}$.
By (\ref{conv}) for sufficiently big $n$ we have 
$$\mathbf{s}(f^n(y_1)-f^n(y_2))\leq\varepsilon$$
for every $\mathbf{s}\in \mathbf{E}$. 
Hence $f^n(y_2)\in V(f^n(y_1));\mathbf{E},\varepsilon)$ and $\{f^n(y_1),f^n(y_2)\}\subseteq U$. Thus $f$ is a weak topological contraction on $Y$. As it is also, by the hypothesis, continuous, by Theorem \ref{T2-continuous-fixed} there exists a unique fixed point for $f$.
\end{proof} 

Let us state the following corollary to the above theorem.

\begin{cor}\label{general-lip}
If $X$ is a linear topological space, $\mathbf{S}$ is a family of seminorms on $X$ separating points in $X$, $Y\subseteq X$ is compact in the topology $\tau$ generated by $\mathbf{S}$, $\phi:[0,\infty)\to[0,\infty)$ and $\phi^n\to0$ ($\phi^n$ denotes here the $n$th iteration of $\phi$), $\phi(t)<t$ for $t\in(0,\infty)$, and $f:Y\to Y$ satisfies
\begin{equation}\label{lipschitz-gen}
\mathbf{s}(f(x)-f(y))\leq\phi(\mathbf{s}(x-y))
\end{equation}
for each $\mathbf{s}\in\mathbf{S}$  and each $x,y\in Y$, then $f$ has a unique fixed point in $Y$. 
\end{cor} 
\begin{proof}
First let us notice that (\ref{lipschitz-gen}) implies the continuity of $f$ with respect to the topology $\tau$. 
Thus it is enough to check that (\ref{conv}) holds. By (\ref{lipschitz-gen}) we have
$$\mathbf{s}(f^n(x)-f^n(y))\leq\phi(\mathbf{s}(f^{n-1}(x)-f^{n-1}(y))\leq\phi^2(\mathbf{s}(f^{n-2}(x)-f^{n-2}(y))\ldots$$
$$\leq\phi^n(\mathbf{s}(x-y))
\to0.$$
\end{proof}

Theorem \ref{conv-fixed} applies to natural sets $Y$ in weak*  topologies. Namely, the following theorem holds.

\begin{thm}\label{weak*}
Let $X$ be a linear topological  space. Let $V$ be a neighbourhood of the zero vector in $X$. We define $Y$ as
$$Y:=\{x^*\in X^*:|x(x^*)|\leq 1,\;\; \textnormal{for}\;\;\textnormal{each}\;\;x\in U\}$$
(we use the dual notation: $x(x^*):=x^*(x)$ for functionals $x^*$ which are members of $X^*$ and elements $x$ of the space $X$). 
Let $f:Y\to Y$ be a weak*-continuous mapping satisfying
$$\lim_n|z(f^n(x^*)-f^n(y^*))|=0$$
for every $z\in X$ 

Then $f$ has a unique fixed point in $Y$.
\end{thm}

\begin{proof} By the Banach-Alaoglu theorem the set $Y$ is compact in the weak* topology on $X^*$ (comp. \cite{R}). Substituting  $X^*$ for $X$ and the seminorms $\{|x(\cdot)|:x\in X\}$ 
for $\mathbf{S}$ in Theorem \ref{conv-fixed} we obtain the conclusion.  

\end{proof}

We used Theorem \ref{T2-continuous-fixed} and Theorem \ref{LN} to prove Theorem \ref{conv-fixed} which deals with continuous mappings. We can use in an analogous way Theorem \ref{T2-closed-fixed} and again Theorem \ref{LN} to prove a conterpart of  Theorem \ref{conv-fixed} for closed mappings (the proof is fully analogous to the proof of Theorem \ref{conv-fixed}).

\begin{thm}\label{conv-fixed-closed} If $X$ is a linear topological space, $\mathbf{S}$ is a countable family of seminorms on $X$ separating points in $X$, $Y\subseteq X$ is compact in the topology $\tau$ generated by $\mathbf{S}$, and $f:Y\to Y$ is a closed mapping in the topology $\tau$ such that
$$\lim_n\mathbf{s}(f^n(x)-f^n(y))=0$$
for each $\mathbf{s}\in\mathbf{S}$  and each $x,y\in Y$, then $f$ has a unique fixed point in $Y$. 
\end{thm}

The next theorem is an analogue of Lebesgue's number lemma for compact topological monoids (i.e. semigroups with neutral element).

\begin{thm}\label{semigroup}
If $G$ is a compact Hausdorff topological semigroup with a neutral element, then for any open cover of $G$ there exists an open neighbourhood of the neutral element whose each left translation is contained in some element of the cover.
\end{thm} 

\begin{proof} Because $G$ is compact we can consider a finite open cover $\mathcal{U}=\{U_1,\ldots,U_n\}$. Assume now that for each open neighbourhood $V$ of the neutral element there exists a translation $x_VV$ that is not contained in any $U_i$. Because $G$ is compact, there exists a subnet $(x_U)_U$ of the net $(x_V)_V$ which is convergent to some point $x$. 
Let $x\in U_i$.  
By continuity of the semigroup operation 
there are open neighbourhoods $A$ of $x$ and $B$ of the neutral element such that $AB\subseteq U_i$. 
Let $x_U\in A$ and $U\subseteq B$. 
Then $x_UU\subseteq 
U_i$ which is a contradiction. 

\end{proof}

We will use this theorem to prove a fixed point theorem for compact semigroups.

\begin{thm}\label{semigroup-fixed}
If $G$ is a Hausdorff compact topological monoid and $f:G\to G$ is a continuous mapping such that for each $x,y\in G$ and each neighbourhood $V$ of the neutral element there exist $z\in G$ and $n\in\mathbb{N}$ such that $f^n(x),f^n(y)\in zV$, then $f$ has a unique fixed point.
\end{thm}

\begin{proof}
Let $x,y\in G$. Let $\mathcal{U}$ be an open cover of $G$. Let $z\in G$ and $n\in\mathbb{N}$ be chosen for $x,y$ according to the hypothesis. By Theorem \ref{semigroup} there exists an open neighbourhood $V$ of the neutral element such that $zV$ is contained in some $U\in\mathcal{U}$, and thus $f^n(x),f^n(y)\in U$. Hence $f$ is a weak topological contraction and by Theorem \ref{T2-continuous-fixed} it has a unique fixed point.   
\end{proof}

We used Theorem \ref{T2-continuous-fixed} and Theorem \ref{semigroup} to prove Theorem \ref{semigroup-fixed} which deals with continuous mappings. One can use in an analogous way Theorem \ref{T2-closed-fixed} and again Theorem \ref{semigroup} to prove the conterpart of  Theorem \ref{semigroup-fixed} for closed mappings (again the proof is fully analogous to the proof of Theorem \ref{semigroup-fixed}).

\begin{thm}\label{semigroup-fixed-closed}
If $G$ is a first countable Hausdorff compact topological monoid and $f:G\to G$ is a closed mapping such that for each $x,y\in G$ and each neighbourhood $V$ of the neutral
 element there exist $z\in G$ and $n\in\mathbb{N}$ such that $f^n(x),f^n(y)\in zV$, then $f$ has a unique fixed point.
\end{thm}

\section{Locally Hausdorff $T_1$ spaces}\label{loc} 

We can strengthen Theorem \ref{T2-continuous-fixed} replacing the assumption that the space is Hausdorff by the assumption that the space is $T_1$ locally Hausdorff.

\begin{thm}\label{loc-T2-continuous-fixed}
If $X$ is a locally Hausdorff $T_1$ topological space and $f$ is a continuous weak topological contraction on $X$, then $f$ has a unique fixed point.
\end{thm}

\begin{proof} 
Let us assume that the conclusion does not hold. Let $x\in X$. As in the proof of Theorem \ref{T2-continuous-fixed} we find a point $y$ such that for every open neighbourhood $U$ of $y$ there are infinitely many $n$'s such that $\{f^n(x),f^{n+1}(x)\}\subseteq U$.
For each  neighbourhood $U$ of $y$ let us choose $n_U$ such that  $\{f^{n_U}(x),f^{{n_U}+1}(x)\}\subseteq U$.
Because, as we have assumed, $f$ does not have a fixed point, we have $f(y)\neq y$. 

Let now $U$ be a neighbourhood of $y$ and $W$ be a neighbourhood of $f(y)$. Let $V=U\cap f^{-1}[W]$. We have $\{f^{n_V}(x),f^{{n_V}+1}(x)\}\subseteq V\subseteq U$ and  $f^{{n_V}+1}(x)=
f(f^{n_V}(x))\in W$. Hence $U\cap V\neq\emptyset$. Thus $y$ and $f(y)$ cannot be separated by any of their open neighbourhoods.

To every neighbourhood $H$ of $f(y)$ belong points 
$f(f^{n_{f^{-1}[H]}}(x))=f^{n_{f^{-1}[H]}+1}(x)$ and $f(f^{n_{f^{-1}[H]}+1}(x))
=f^{n_{f^{-1}[H]}+2}(x)$. As above for $y$ and $f(y)$, we can show now that the points $f(y)$ and $f^2(y)=f(f(y))$ cannot be separated by their open neighbourhoods.

Inductively, knowing that to every neighbourhood $G$ of $f^{n}(y)$ belong two points of the form $f^m(x),f^{m+1}(x)$ we infer that the same is true for  $f^{n+1}(y)$ and that the points $f^{n}(y)$, $f^{n+1}(y)$ cannot be separated by their open neighbourhoods.

Let $z\in X$. Let  $U(z)$ be a Hausdorff open neighbourhood of $z$. The family 
$\mathcal{U}=\{U(z):z\in X\}$
is an open cover of $X$ and, because   the points $f^{n}(y)$, $f^{n+1}(y)$ cannot be separated by their open neighbourhoods, if $f^n(y)\in U(z)$ then $f^{n+1}(y)\notin U(z)$, but this contradicts the assumption that $f$ is a week contraction.

Assume that $y\neq z$ and $f(y)=y$ and $f(z)=z$. Then for the open cover $\{X\setminus\{y\},X\setminus\{z\}\}$ for no $n$ the pair of points $f^n(y)=y,f^n(z)=z$ is in one element of the cover, which is a contradiction with the assumpion that $f$ is a week contraction. 

\end{proof}

We can formulate a version of Theorem \ref{loc-T2-continuous-fixed} with a weaker hypothesis and a conclusion without unuiqueness of a fixed point. Except for uniqueness the proof goes along the same lines.

\begin{thm}
If $X$ is a locally Hausdorff space and for every open cover of $X$ and every $x\in X$ there exists $n\in\mathbb{N}_0$ such that
$f^n(x),f^{n+1}(x)$ are members of the same element of the cover, then $f$ has a fixed point.  
\end{thm}


In view of Theorem \ref{loc-T2-continuous-fixed} one could expect that also Theorem \ref{T2-closed-fixed} can be generalized to the case of first countable locally Hausdorff $T_1$ spaces. This, however, is not the case. We give below an example of a $T_1$ locally Hausdorff space and a closed weak contraction on this space that does not have a fixed point.   
 
\vspace{0.4 cm} 

\noindent {\bf Example 3.} Let $X=\mathbb{Z}$. The base for the topology on $X$ is defined as follows. For all points  
$n\in \mathbb{N}$ the set $\{n\}$ is open. The open neighbourhood base for $-n$, for $n\in \mathbb{N}_0$, consists of sets $A\cup\{-n\}$ where $A$ is a co-finite subset of $\mathbb{N}$. Let $f:\mathbb{Z}\to\mathbb{Z}$ be defined by: $f(m)=m+1$ for all $m\in\mathbb{Z}$.

For every $n\in\mathbb{N}_0$ the set $\{-n\}\cup\mathbb{N}$ is open and it is a Hausdorff space with the topology inherited from the whole space. Thus $\mathbb{Z}$ is locally Hausdorff.

Let $a,b\in\mathbb{Z}$. Let $\{U_\alpha\}_\alpha$ be an open cover of $\mathbb{Z}$. Let $0\in U_\alpha$. Then $\mathbb{N}\cap U_\alpha$ is a co-finite subset of $\mathbb{N}$ and for some $m$ big enough $f^m(a)=a+m\in U_\alpha,f^m(b)=b+m\in U_\alpha.$ Thus $f$ is a weak contraction.

If a set $E\subseteq\mathbb{Z}$ is closed and $E\cap (-\mathbb{N}_0)\neq-\mathbb{N}_0$ then $E\cap\mathbb{N}$ must be a finite set.
Actually, every set $E$ of the form $E=A\cup B$ where $A$ is a finite subset of $\mathbb{N}$ and $B$ is any subset of $-\mathbb{N}_0$,  is closed. In this case $f[E]=(A+1)\cup(B+1)$ and thus $|f[E]\cap\mathbb{N}|\leq|A|+1$ (we have equality here if $0\in B$). Thus $f[E]$ is also closed.

If a set $E\subseteq\mathbb{Z}$ and $E\cap (-\mathbb{N}_0)=-\mathbb{N}_0$ then $E$ is a closed set. Then also $f[E]\cap (-\mathbb{N}_0)=(E+1)\cap(-\mathbb{N}_0)=-\mathbb{N}_0$ and $f[E]$ is closed.

Thus the mapping $f$ is closed. 

We complete our example by noticing that $f$ does not have a fixed point.  

\section{Peripherally Hausdorff spaces}\label{alfa-H}

In Section \ref{terminology} we introduced the notion of an $\alpha$-Hausdorff space, $\alpha\in\textrm{On}$. Here we present some properties of these spaces.

Let us first introduce one more symbol.  Let $Y$ be a $T_1$ topological space and $X\subseteq Y$.
If $u$ is an element of $X$, the class $[u]$ in the space $X$ with the relative topology inherited from $Y$ will be sometimes denoted by $[u]_X$, if $X$ may not be clearly distinguishable from the context. 

\begin{thm}\label{alfa-HH}
For every $\alpha\in\textnormal{On}$ there exists a peripherally Hausdorff  space $X$ such that
$$\textnormal{rank}_{T_2}(X)=\alpha.$$ 
\end{thm}

\begin{proof} The proof goes by induction with respect to $\alpha\in\textrm{On}$. For $\alpha=0$ let $X$ be any Hausdorff space.

Now let us assume that the conclusion of the theorem holds for some $\alpha$. Let $\textnormal{rank}_{T_2}(X)=\alpha$ for some topological space $X$. Let $\{y_0, y_1,y_2\ldots\}\cap X=\emptyset$ and $y_i\neq y_j$ for $i\neq j$. Let $Y=X\cup\{y_i:i\in\mathbb{N}_0\}.$ Let us define a topology on $Y$. Let a local base at any point $x\in X$ consist of the sets of the form $U\cup A$ where $A$ is any co-finite subset of $\{y_1,y_2,\ldots\}$ and $U$ is an open neighbourhood of $x$ in $X$. Let a local neighbourhood base of $y_0$ consist of sets $A\cup\{y_0\}$ where again  $A$ is any co-finite subset of $\{y_1,y_2,\ldots\}$.
Let singletons $\{y_1\}, \{y_2\},\ldots$ be open sets.  For $x\in X$:
$$[x]_Y=X\cup\{y_0\}$$
and 
$$[y_0]_Y=X\cup\{y_0\},$$
$$[y_i]_Y=\{y_i\},$$
$i=1,2,\ldots.$
As in the space $X\cup\{y_0\}$ the singleton $\{y_0\}$ is clopen the Hausdorff rank of $[x]_Y$ is equal to the Hausdorff rank of $X$, namely to $\alpha$. Hence 
$$\textrm{rank}_{T_2}(Y)=\alpha+1.$$

Let now $\gamma$ be a limit ordinal and we assume that for all $\alpha<\gamma$ there exists a space $X_\alpha$ whose Hausdorff rank is equal to $\alpha$. Then let $Y$ be a topological sum of the spaces $X_\alpha$, $\alpha<\gamma$. Then for $y\in X_\alpha$ we have
$$[y]_Y=[y]_{X_\alpha}.$$ Hence 
$$\textrm{rank}_{T_2}(Y)=\gamma.$$   

\end{proof}

We now give an example of a $1$-Hausdorff space where the classes $[x]$ have a natural geometric meaning.

\vspace{0.4 cm}

Let $X=\mathbb{R}^2.$ For $(x,y)$ a local neighbourhood base is defined as the family of sets of the form 
$(U\times\{0\})\cup(\{x\}\times V)$, where $U$ and $V$ are neighbourhoods of $x$ and $y$ on $\mathbb{R}$, respectively.
Then points $(x,y)$ and $(x,v)$, $y\neq v$, cannot be separated from each other by disjoint open neighbourhoods, while $(x,y)$ and  $(u,v)$, where $v\neq x$ can be separated in that way. Hence $[(x,y)]=\{x\}\times\mathbb{R}$. For the vertical line $\{x\}\times\mathbb{R}$  the topology inherited from $X$ is the usual Euclidean topology. Thus $X$ is $1$-Hausdorff.

\vspace{0.4 cm}
 
The next theorem says that the class of peripherally Hausdorff spaces is closed under finite Cartesian products.

\vspace{0.4 cm}

\begin{thm}\label{alfa-product}
If $X$ and $Y$ are $T_1$ topological spaces and $\textnormal{rank}_{T_2}(X)=\alpha$ and $\textnormal{rank}_{T_2}(Y)=\beta$, then 
$\textnormal{rank}_{T_2}(X\times Y)=\max(\alpha,\beta)$.
\end{thm}

\begin{proof}
First let us notice that
\begin{equation}\label{product-class}
[(x,y)]_{X\times Y}=[x]_X\times[y]_Y.
\end{equation}
The proof goes by transfinite induction with respect to $\alpha\in\textrm{On}$.
Let $\alpha=0$, i.e. we assume that $X$ is Hausdorff. Let $\textnormal{rank}_{T_2}(Y)=\beta$.
By the definition of the product topology we have for every $x\in X$ and every $y\in Y$:
$$
[(x,y)]_{X\times Y}=[x]_X\times[y]_Y=\{x\}\times[y]_Y,
$$
and, as the product $\{x\}\times[y]_Y$ is homeomorphic with $[y]_Y$ we obtain  
$$\textnormal{rank}_{T_2}(X\times Y)=\beta=\max(0,\beta)=\max(\alpha,\beta).$$

Now assume that the conclusion is true for some $\alpha$. By  (\ref{product-class}) and, as $\textrm{rank}_{T_2}([x]_{X})<\alpha$, by the induction hypothesis we obtain 
\begin{equation}\label{ranks-of-product-classes}
\textrm{rank}_{T_2}([(x,y)]_{X\times Y})=\max(\textrm{rank}_{T_2}([x]_{X},\textrm{rank}_{T_2}([y]_{Y}).
\end{equation}
Hence 
$$\textrm{rank}_{T_2}(X\times Y)>\textrm{rank}_{T_2}([x]_{X})$$
for every $x\in X$ and
$$\textrm{rank}_{T_2}((X\times Y)>\textrm{rank}_{T_2}([y]_{Y}).$$
Therefore
$$\textrm{rank}_{T_2}(X\times Y)\geq\textrm{rank}_{T_2}(X)$$
and
$$\textrm{rank}_{T_2}(X\times Y)\geq\textrm{rank}_{T_2}(Y).$$
Hence
\begin{equation}\label{ineq1} 
\textrm{rank}_{T_2}(X\times Y)\geq\max(\textrm{rank}_{T_2}(X),\textrm{rank}_{T_2}(Y)).
\end{equation}
On the other hand for each $x\in X$ and each $y\in Y$:
$$\max(\textrm{rank}_{T_2}(X),\textrm{rank}_{T_2}(Y))>\max(\textrm{rank}_{T_2}([x]_{X},\textrm{rank}_{T_2}([y]_{Y}),$$
and by (\ref{ranks-of-product-classes})
\begin{equation}\label{ineq2}
\max(\textrm{rank}_{T_2}(X),\textrm{rank}_{T_2}(Y))\geq \textrm{rank}_{T_2}(X\times Y).
\end{equation} 
Inequalities (\ref{ineq1}) and (\ref{ineq2}) conclude the proof.

\end{proof}

This theorem does not have generalization to infinite products. In fact, a product of countably many peripherally Hausdorff spaces may not be peripherally Hausdorff. Indeed, if $\textrm{rank}_{T_2}(X_n)=n$ for $n\in\mathbb{N}$, then  the product of $X_n$s is not  peripherally Hausdorff. Striving for a contradiction assume that   
$$\textrm{rank}_{T_2}\left(\prod_{n}X_n\right)=\alpha$$
for some $\alpha\in\textrm{On}$. Let $x^{(1)}_n\in X_n$ and $\textrm{rank}_{T_2}([x^{(1)}_n]_{X_n})=n-1$. We have
$$[(x^{(1)}_n)_n]_{\prod_{n}X_n}=\prod_{n}[x^{(1)}_n]_{X_n}$$
and
$$\textrm{rank}_{T_2}([(x^{(1)}_n)_n]_{\prod_{n}X_n})=\alpha_1<\alpha.$$ 
Let now $x^{(2)}_n\in[x^{(1)}_n]$ and 
$$\textrm{rank}_{T_2}([x^{(2)}_1]_{[x_1]})=0,\;\textrm{and}\;\textrm{rank}_{T_2}([x^{(2)}_m]_{[x^{(1)}_m]})=m-2, \textrm{for}\;m\geq 2.$$
We have
$$\textrm{rank}_{T_2}([(x^{(2)}_n)_n]_{\prod_{n}[x^{(1)}_n]})=\alpha_2<\alpha_1.$$  
Continuing in that way we obtain a sequence of points $((x^{(k)}_n)_n)_k$, such that 
$$x^{(k)}_n\in [x^{(k-1)}_n],$$
$$\textrm{rank}_{T_2}([(x^{(k)}_n)_n]_{\prod_{n}[x^{(k-1)}_n]})=\alpha_k<\alpha_{k-1}$$
and
$$\textrm{rank}_{T_2}([x^{(k)}_1]_{[x^{(k-1)}_1]})=\ldots=\textrm{rank}_{T_2}([x^{(k)}_{k-1}]_{[x^{(k-1)}_{k-1}]})=0,$$
and
$$\textrm{rank}_{T_2}([x^{(2)}_m]_{[x^{(2)}_m]})=m-k, \textrm{for}\;m\geq k.$$
Hence none of the sets $[(x^{(k)}_n)_n]_{\prod_{n}[x^{(k-1)}_n]}$ has the Hausdorff rank equal to zero and we obtain the infinite decreasing sequence of ordinals:
$$\alpha>\alpha_1>\alpha_2>\ldots,$$
which is impossible.

The following theorem gives a sufficient condition for a subspace of a $T_1$ locally Hausdorff space to be Hausdorff. Corollaries to this theorem provide sufficient conditions for a  a $T_1$ locally Hausdorff space to be peripherally Hausdorff.

\begin{thm}\label{local-setT2}
If a $T_1$ topological space $X$ is locally Hausdorff, $Y\subseteq X$ and $Y\subseteq[u]\cup[v]$  for every two points $u,v\in Y$, then $Y$
as a subspace of $X$ is a Hausdorff space.
\end{thm}

\begin{proof} Let $u,v\in Y$. Let $U,V$ be Hausdorff neighbourhoods of $u$ and $v$, respectively. Because $X$ is $T_1$ we can assume that $u\notin V$ and $v\notin U$. Because $[u]$ consists of all points of $X$ that cannot be separated from $u$ by disjoint open neighbourhoods no point from $[u]\setminus\{u\}$ is in $U$. Similarly no point from $[v]\setminus\{v\}$ is in $V$. Thus no point from $[u]\cup[v]$ is in $U\cap V$, and, because $Y\subseteq[u]\cup[v]$, no point from $Y$ is in $U\cap V$.  
\end{proof}  

The following fact is an immediate corollary to the previous theorem.

\begin{cor}\label{local-to-peripheral1} If a $T_1$ topological space $X$ is locally Hausdorff, and $[x]\subseteq[u]\cup[v]$ for every $x\in X$ and every two points of $u,v\in [x]$,  then $X$ is peripherally Hausdorff, more exactly: a $1$-Hausdorff space.
\end{cor}

The next corollary covers a more special, though it seems a natural, case. 

\begin{cor}\label{local-to-peripheral2} If a $T_1$ topological space $X$ is locally Hausdorff, and classes $[x]$ either coincide or are disjoint (so they are equivalence classes) then  $X$ is  peripherally Hausdorff, more exactly: a $1$-Hausdorff space.
\end{cor}

With a little stronger assumption about a contraction we can prove a fixed point theorem also for the peripherally Hausdorff spaces.

\begin{thm}\label{peripheral-fixed} If $X$ is a peripherally Hausdorff space, i.e. it is an $\alpha$-Hausdorff space for some $\alpha\in\textnormal{On}$, and $f:X\to X$ is a continuous weak+ topological contraction, then $f$ has a unique fixed point.
\end{thm}

\begin{proof} We prove the theorem by a transfinite induction with respect to $\alpha\in\textrm{On}$.

If $\alpha=0$ then $X$ is a Hausdorff space and we apply Theorem \ref{T2-continuous-fixed}.

Now assume that the conclusion holds for every ordinal number $\beta<\alpha$. If $X$ is an $\alpha$-Hausdorff space
and $f:X\to X$ is a weak+ contraction and $x\in X$ then there exists $y\in X$ such that for each neighbourhood $U$ of $y$
there exists $n\in\mathbb{N}$ such that $\{f^n(x),f^{n+1}(x),\ldots\}\subseteq U$ (for otherwise for each $u\in X$ there would exist
its neighbourhood $V(u)$ without this property, and for the open cover $\{V(u):u\in X\}$ of $X$ the mapping $f$ and would not satisfy the definition of a weak+ topological contraction for the pair of points $x,f(x)$). 

\vspace{0.2 cm}

\noindent Claim. $\{f^n(y):n\in\mathbb{N}_0\}\subseteq[y]$.  

\vspace{0.2 cm}

\noindent Proof of Claim 1. If $f(y)=y$ the claim is evidently true. So assume that $f(y)\neq y$. By the continuity of $f$ for every neighbourhood $V$ of $f(y)$ almost all elements of the sequence $f(x),f^2(x),\ldots$ lie in $V$. Hence $f(y)$ and $y$ cannot be separated  by two disjoit open sets and, consequently, $f(y)\in[y]$. If now $f(y)=f^2(y)=f(f(y))$ then the claim is again evidently true.
So assume that $f^2(y)\neq f(y)$. By continuity of $f$ for every neighbourhood $V$ of $f^2(y)$ almost all elements of the sequence $f(x),f^2(x),\ldots$ lie in $V$, which shows that $f^2(y)$ and $y$ cannot be separated by disjoint open sets. Hence $f^2(y)\in[y].$ We proceed by induction proving that $f^n(y)\in [y]$ for all $n\in\mathbb{N}_0$. This ends the proof of the claim. 

\hspace*{\fill}$\diamond$
 
\vspace{0.2 cm}

Let $R:=\textrm{cl}(\{f^n(y):n\in\mathbb{N}_0\})$. By the continuity of $f$ we have  $f[R]\subseteq R$. Because $[y]$ is a closed set we have $R\subseteq[y]$. By the definition of an $\alpha$-Hausdorff space the set $[y]$ is a $\beta$-Hausdorff space for some $\beta<\alpha$ and then also $R$, as a subset of $[y]$, is a $\beta$-Hausdorff space. Thus, by our induction hypothesis, there exists $r\in R$ such that $f(r)=r.$ 

It remains to prove that $r$ is a unique fixed point of $f$. If $r\neq s\in X$ and $f(s)=s$ then for the open cover $\{X\setminus\{r\},X\setminus\{s\}$ and the pair of points $r,s$ $f$ does not satisfy the definition of a weak+ contraction.

This concludes the proof of the theorem

\end{proof}

\section{Hutchinson operators in $T_1$ compact spaces}

In \cite{MMRR} the following theorems were proven (Theorem 6, and Corollary 7 therein, resp.).

\vspace{0.4 cm}

\noindent{\bf Theorem B.} {\it 
Let $X$ be a $T_1$ compact space. Let $\mathcal{F}$ be a contractive IFS on $X$. Then, if the Hutchinson operator $F$ induced by $\mathcal{F}$ is closed, then $F$ is a topological contraction on $2^X$.} 

\vspace{0.4 cm}

\noindent{\bf Theorem C.} {\it
If $X$ is a $T_1$ compact space then every contractive IFS inducing a closed Hutchinson operator has a unique attractor; in other words: the Hutchinson operator induced by this IFS has a unique fixed point.}

\vspace{0.4 cm}

Of course, Theorem C follows from Theorems A and B.

The attractor in the conclusion of Theorem C is unique. One should add that the very existence of an attractor for an IFS on $T_1$ compact space does not require any additional assumptions imposed on the IFS (for instance, that it is contractive). Namely the following theorem (Theorem 5 
in \cite{MMRR}) holds.

\vspace{0.4 cm}

\noindent{\bf Theorem D.} {\it
If $X$ is a $T_1$ compact space then any IFS on $X$ has an attractor; in other words the  Hutchinson operator $F$ induced by $\mathcal{F}$ has a fixed point.}  

\vspace{0.4 cm}

There arises a question whether the Hutchinson operator induced by a contractive IFS is always closed.
We  shall construct an example of a $T_1$ compact topological space and a contractive IFS on this space inducing a Hutchinson operator which is not closed. The existence of such a space and an IFS we state as a theorem.

\begin{thm}\label{Hutchinson-example}
There exists a $T_1$ compact space $X$ and a contractive IFS inducing a Hutchinson operator  which is not closed (as a mapping from $2^X$ to $2^X$). 
\end{thm}

\begin{proof} Let 
$$\textrm{ODD}=\{1,3,5,\ldots\}$$
and
$$\textrm{EVEN}=\{2,4,6,\ldots\}.$$
Let 
$$X=\mathbb{N}\cup\{a,b\}$$
where $a\neq b$, $a,b\notin\mathbb{N}$, and the set $A\subseteq X$ is open if:

\vspace{0.2 cm}

\noindent i) $A\subset \textrm{ODD}$,

\vspace{0.2 cm}

or  

\vspace{0.2 cm}

\noindent ii) $A\subseteq\mathbb{N}$ and the set $\textrm{ODD}\setminus A$ is finite,

\vspace{0.2 cm}

or  

\vspace{0.2 cm}

\noindent iii) $A\cap\{a,b\}\neq\emptyset$ and $A$ is a co-finite set in $X$.

\vspace{0.2 cm}

We define an IFS $\mathcal{F}$ as 
$$\mathcal{F}=\{f,g\},$$
where 
$$f(x)=\left\{
\begin{array}{ccc}
a&\textrm{if}&x=a,b,\\
b&\textrm{if}&x=2n,n\in\mathbb{N},\\
2n&\textrm{if}&x=2n-1,n\in\mathbb{N},\\
\end{array} 
\right.
$$
and
$$g(x)=\left\{
\begin{array}{ccc}
b&\textrm{if}&x=a,b,\\
a&\textrm{if}&x=2n,n\in\mathbb{N},\\
2n&\textrm{if}&x=2n-1,n\in\mathbb{N}.
\\
\end{array} 
\right.
$$

\vspace{0.2 cm}

Let $E\subseteq X$ be a nonempty closed set. If $E$ is finite,
then $f[E]$ is also finite and, therefore, closed. 
If $E$ is infinite, then $a,b\in E$ and $E\cap\textrm{EVEN}\neq\emptyset$.
Then $\{a,b\}\subseteq f[E]\subseteq\{a,b\}\cup\textrm{EVEN}$ which is a closed set. 

Thus $f$ is a closed mapping. Analogously, one argues that $g$ is closed.

Now let us consider a composition $h_n\circ h_{n-1}\circ\ldots\circ h_2\circ h_1$, where each $h_i$ is equal to either $f$ or $g$.
Because $f[X]=g[X]=\textrm{EVEN}\cup\{a,b\}$ we have $h_1[X]=\textrm{EVEN}\cup\{a,b\}$. Because 
$$f[\textrm{EVEN}\cup\{a,b\}]=g[\textrm{EVEN}\cup\{a,b\}]=\{a,b\}$$
we have 
$$h_2\circ h_1[X]=\{a,b\}.$$
Because $f[\{a,b\}]=\{a\}$ and $g[\{a,b\}]=\{b\}$ we have for $n\geq3$
$$h_n\circ h_{n-1}\circ\ldots\circ h_2\circ h_1[X]=\{a\}$$
or
$$h_n\circ h_{n-1}\circ\ldots\circ h_2\circ h_1[X]=\{b\},$$
hence the IFS $\mathcal{F}=\{f,g\}$ is contractive.

Let $F$ be the Hutchinson operator induced by $\mathcal{F}$, namely $F:2^X\to2^X$ is defined as 
$$F(E):=f[E]\cup g[E].$$
We shall show that $F$ is not closed as a mapping from $2^X$ to $2^X$.

The one point set $\{a\}$ is not in the image $F[2^X]$ of the hyperspace $2^X$ via $F$. 
Thus it is enough to show that $\{a\}\in\overline{F[2^X]}$. Let $S(V_0;V_1,\ldots,V_k)$ be a base neigbourhood of $\{a\}$ which simply means here that $a\in\bigcap_{i=0}^kV_i$. Thus the set $\bigcap_{i=0}^kV_i$ is an open neighbourhood of $a$. Hence it must be a co-finite subset of $X$ and $2n\in\bigcap_{i=0}^kV_i$, for some $n\in\mathbb{N}$. This implies that $\{2n\}\in S(V_0;V_1,\ldots,V_k)$. We also have $\{2n\}=F(\{2n-1\})$. Hence $S(V_0;V_1,\ldots,V_k)\cap F[2^X]\neq\emptyset$. We conclude that $\{a\}\in \overline{F[2^X]}\setminus F[2^X]$.  

\end{proof}

We have shown that the Hutchinson operator of a contractive IFS on a compact $T_1$ space may not be closed and then it is not a topological contraction, because this is one of the conditions defining topological contraction. To apply directly Theorem A we assumed in \cite{MMRR} in the hypothesis of Theorem C that the Hutchinson operator was closed. It turns out, however, that to get the conclusion about the unique fixed point of the Hutchinson operator induced by a contractive IFS the assumption that the Hutchinson operator is closed is not necessary. Namely, we prove here the following stronger version of Theorem C.

\begin{thm}\label{attractor2}
If $X$ is a $T_1$ compact space then every contractive IFS has a unique attractor; in other words: the Hutchinson operator induced by this IFS has a unique fixed point.
\end{thm}

We present two proofs of this theorem. In both proofs the existence of at least one fixed point of the Hutchinson operator induced by any IFS (\cite{MMRR}) is used  but the first proof is rather direct; the second and shorter one depends more heavily on some facts proven in \cite{MMRR}.

\vspace{0.4 cm}

\noindent {\it Proof 1.} Let $\mathcal{F}=\{f_1.\ldots,f_m\}$ be a contractive IFS. Let $F$ be the Hutchinson operator induced by $\mathcal{F}$. By Theorem 5 in \cite{MMRR} $F$ has a fixed point. Let us assume that $E_1,E_2\in2^X$ are two different fixed points of $F$. Let, e.g., $z\in E_1\setminus E_2$.
Let $U=E_2^c$ and $V=\{z\}^c$. We have $U\cup V=X$. Because $\mathcal{F}$ is a contractive IFS there exists $n\in\mathbb{N}$ such that for each sequence $(i_1,\ldots,i_n)\in\{1,\ldots,m\}^n$
$$f_{i_n}\circ\ldots\circ f_{i_1}[X]\subseteq U \;\;\textrm{or}\;\;f_{i_n}\circ\ldots\circ f_{i_1}[X]\subseteq V.$$
Because $E_1$ is a fixed point of $F$ we have $F^n[E_1]=E_1$. Hence there exists $s\in E_1$ and a sequence $(i_1,\ldots,i_n)\in\{1,\ldots,m\}^n$ such that 
$$f_{i_n}\circ\ldots\circ f_{i_1}(s)=z$$
and this implies
$$f_{i_n}\circ\ldots\circ f_{i_1}[X]\subseteq U.$$
Because 
$$f_{i_n}\circ\ldots\circ f_{i_1}[E_2]\subseteq f_{i_n}\circ\ldots\circ f_{i_1}[X]$$
we infer that
$$f_{i_n}\circ\ldots\circ f_{i_1}[E_2]\subseteq U.$$
Because $U\cap E_2=\emptyset$ we obtain 
$$f_{i_n}\circ\ldots\circ f_{i_1}[E_2]\setminus E_2\neq\emptyset$$.
Hence
$$F(E_2)\setminus E_2 =\bigcup_{(j_1,\ldots,j_n)\in\{1,\ldots,m\}^m}f_{j_n}\circ\ldots\circ f_{j_1}\setminus E_2\neq\emptyset,$$
which is an obvious contradiction with $E_2$ being a fixed point of $F$. This concludes the proof.

$\,$\hfill$\Box$

\vspace{0.4 cm}

\noindent{\it Proof 2.} By inspection of the proof of Theorem 4 in \cite{MMRR} (here Theorem B above) one can see that the assumption that the Hutchinson operator was closed was not used to prove that it had property $(\star)$.  

Let $F$ be the Hutchinson operator induced by our IFS $\mathcal{F}$. Let $E_1$ and $E_2$ be two different fixed points of $F$. 
We have $F^n(E_1)=E_1$ and $F^n(E_2)=E_2$, for each $n\in\mathbb{N}$. Hence 
\begin{equation}\label{inclusion}
\{E_1,E_2\}\subseteq F^n[2^X].
\end{equation}
The pair of sets $U=\{E_1\}^c$, $V=\{E_2\}^c$ is an open cover of $2^X$ and, because $F$ satisfies $(\star)$, for some $n$ either $F^n[2^X]\subseteq U$ or $F[2^X]\subseteq V$, but this contradicts (\ref{inclusion}).  

$\,$\hfill$\Box$

{}

\end{document}